\documentclass[11pt]{article}
\usepackage[hmargin=0.85in,vmargin=0.68in]{geometry}
\usepackage{amsmath,amssymb,amsthm,mathtools}
\usepackage{tikz-cd}
\usepackage{xurl}
\usepackage[colorlinks=true,linkcolor=blue,citecolor=blue,urlcolor=blue]{hyperref}
\hypersetup{pdftitle={Computations with diagonal simplicial tensor modules},pdfauthor={Florian Lengyel}}
\newtheorem{theorem}{Theorem}[section]
\newtheorem{proposition}[theorem]{Proposition}
\newtheorem{corollary}[theorem]{Corollary}
\newtheorem{lemma}[theorem]{Lemma}
\theoremstyle{definition}

\theoremstyle{remark}

\newtheorem{example}[theorem]{Example}
\DeclareMathOperator{\im}{im}
\DeclareMathOperator{\rank}{rank}

\newcommand{\program}[2]{\href{https://github.com/flengyel/SimplicialTensors/blob/c6802c6692215b688864dc18cd9ff3073b804d71/#1}{\nolinkurl{#2}}}
\newcommand{\SII}[2]{\genfrac{\{}{\}}{0pt}{}{#1}{#2}}
\title{Computations with diagonal simplicial tensor modules}
\author{Florian Lengyel\\\small\href{mailto:flengyel@gradcenter.cuny.edu}{\nolinkurl{flengyel@gradcenter.cuny.edu}}}
\date{}
\begin{document}
\maketitle
\begin{abstract}
For each vertex order $\pi$ of a finite simple graph on at least two vertices, insert an isolated vertex between successive vertices.
In the resulting adjacency matrix $A$, deleting row $i$ and column $i$ gives a principal submatrix $A_i$. The alternating sum of the $A_i$ is the simplicial boundary of $A$ and has full rank over $\mathbb F_2$ exactly when $\pi$ traces a Hamiltonian path.
Each Hamiltonian path gives two full-rank orders.
We tabulate boundary rank distributions over all vertex orders of the Petersen, $4\times4$ rook, and Shrikhande graphs.
Full-rank counts distinguish the cospectral rook and Shrikhande graphs.
For each finite graph, vertex and edge tensors in a three-axis diagonal simplicial tensor module generate a simplicial submodule whose normalized chains form the graph's incidence complex and compute its homology.

Let $X$ be a diagonal simplicial tensor module on $k\ge1$ axes over a nonzero unital commutative ring. We give coordinate bases for its horn kernels and normalized chains.
A horn has a unique filler exactly when its degree exceeds $k$. For a shape in $X$ of simplicial dimension $n$ (shortest axis length $n+1$), the simplicial module $X$ is an $n$-hypergroupoid exactly when $k\le n$.
We prove analogous bounds for horn filling in a dendroidal module on rooted trees.
Over a field, we construct a rational map from the space of hypermatrices of fixed shape and degree to a product of Grassmannians. The reduced closure of its image is rational and, over $\mathbb C$, is isomorphic to a Schubert variety.
Python code for this paper's examples and rank tables is at \url{https://github.com/flengyel/SimplicialTensors}.
\end{abstract}
\setlength{\abovedisplayskip}{6pt plus 2pt minus 3pt}
\setlength{\belowdisplayskip}{6pt plus 2pt minus 3pt}
\setlength{\abovedisplayshortskip}{1pt plus 2pt}
\setlength{\belowdisplayshortskip}{4pt plus 2pt minus 2pt}

\section{Introduction}
In a diagonal simplicial tensor module, a simplex is a hypermatrix.
The face $d_i$ deletes index $i$ on every axis, and the degeneracy $s_i$ repeats it.
For an adjacency matrix, $d_i$ deletes row $i$ and column $i$.
Take a finite simple graph on at least two vertices, choose an order $\pi$ of its vertices, and insert one isolated vertex between each pair of successive vertices.
Sum the faces of the resulting adjacency matrix over $\mathbb F_2$.
Proposition~\ref{prop:ordered-graph-ranks} shows that this sum is nonsingular exactly when $\pi$ traces a Hamiltonian path.

Which entries of a hypermatrix can be recovered from all but one of its faces?
An entry survives in the $i$th face precisely when none of its indices equals $i$.
Section~\ref{sec:construction} defines the simplicial operators for hypermatrices of arbitrary shape.
Section~\ref{sec:horn-kernels} uses the surviving entries to describe all fillers of a horn.
Let $X$ be a diagonal simplicial tensor module with $k$ axes. A horn has a unique filler exactly when its degree exceeds $k$.
If a shape in $X$ has shortest axis length $n+1$, its simplicial dimension is $n$. The simplicial module $X$ is then an $n$-hypergroupoid exactly when $k\leq n$.

The same description gives a basis for the normalized chains.
Sections~\ref{sec:normalized} and~\ref{sec:contractions} compute the normalized complex and show that the ambient simplicial module has zero homology in positive degrees.
Its simplicial submodules can have nonzero homology in those degrees.

Over a field, Section~\ref{sec:generated} uses a hypermatrix and its iterated faces to compute the homology of the simplicial submodule it generates.
It also characterizes the possible linear relations among the projections of these simplices to normalized chains.
Tensors assigned to the vertices and edges of a finite graph generate a simplicial submodule whose normalized chains form the incidence complex of the graph.
For the adjacency matrices described above, the maximum boundary rank over all vertex orders determines the minimum number of vertex-disjoint paths covering the vertices.

Section~\ref{sec:tree-horns} defines a dendroidal module of hypermatrices by indexing each axis by the edges of a rooted tree and a fixed finite set of extra labels. Numbering these labels identifies its restriction to linear trees with the diagonal simplicial tensor module.
On trees with at least $k+1$ vertices a horn has at most one filler. On trees with at least $k+2$ vertices every horn fills uniquely. For each $k\geq2$, a branching tree with $k+1$ vertices has an unfillable inner horn. For one axis the sharp bound is two vertices without extra labels and three with them.

For hypermatrices of fixed shape and degree over a field, the spaces of linear relations among the projected faces define a rational map to a product of Grassmannians.
Appendix~\ref{app:kernel-families} proves that the reduced closure of its image is rational over the ground field and, over $\mathbb C$, is isomorphic to a Schubert variety.

The accompanying repository \cite{SimplicialTensors} gives \href{https://github.com/flengyel/SimplicialTensors/blob/c6802c6692215b688864dc18cd9ff3073b804d71/tex/README.md#reproduce-the-calculations}{commands} for reproducing the examples and the boundary rank distributions in Tables~\ref{tab:affine-ranks} and~\ref{tab:interleaved-ranks} of this paper.

\section{The construction}\label{sec:construction}
Let $A$ be a nonzero commutative ring with identity and $k\geq1$. The objects of
$\Delta$ are the ordinals $[p]=\{0,\ldots,p\}$ for $p\geq0$, and its morphisms are the order-preserving maps.

Fix a tuple $\mathbf c=(c_1,\ldots,c_k)$ of nonnegative integers and put $I_q=\prod_a[q+c_a]$.
In degree $q$, axis $a$ has length $q+1+c_a$. Its \emph{excess} $c_a$
counts the indices $q+1,\ldots,q+c_a$ beyond $[q]$.
For $n_1,\ldots,n_k\geq1$, the \emph{simplicial dimension} of the shape
$(n_1,\ldots,n_k)$ is the
largest degree in which it can occur, namely $n=\min_a n_a-1$.
In any degree $0\leq q\leq n$ the shape determines the excesses
$c_a=n_a-q-1$. In particular, $\min_a c_a=0$ when $q=n$.

With $[-1]$ the empty ordinal, put $J_c([p])=[p]\oplus[c-1]=[p+c]$
for $c\geq0$, where $\oplus$ is the ordinal sum \cite[Ch.~VII, \S5, p.~175]{MacLaneCategories}.
For $\alpha:[p]\to[q]$,
$J_c(\alpha)=\alpha\oplus 1_{[c-1]}$ sends $i\in[p]$
to $\alpha(i)$ and $p+t$ to $q+t$ for $1\leq t\leq c$.
The product $I=\prod_a J_{c_a}:\Delta\to\mathbf{Set}$ is a cosimplicial set.
The \emph{diagonal simplicial tensor module} $X=X^{\mathbf c}(A)$ has
$X_p=A^{I_p}$, the module of hypermatrices $T:I_p\to A$.
For real entries this is Lim's description of a hypermatrix \cite[Example~4.5, (4.17)]{Lim2021}.
For $\alpha:[p]\to[q]$, the simplicial operator $X(\alpha):X_q\to X_p$ is precomposition
with $I(\alpha):I_p\to I_q$, so $X(\alpha)(T)=T\circ I(\alpha)$.
In multidegree $(p_1,\ldots,p_k)$ the coordinate bases give an isomorphism of $A$-modules
\[
 A^{[p_1+c_1]}\otimes_A\cdots\otimes_A A^{[p_k+c_k]}
 \ \cong\ A^{\prod_a[p_a+c_a]},\qquad
 e_{m_1}\otimes\cdots\otimes e_{m_k}\longmapsto E_m.
\]
In the $a$th factor, $e_j$ is the $j$th standard basis vector.
The hypermatrix $E_m$ has entry $1_A$ at $m=(m_1,\ldots,m_k)$ and zero elsewhere.
Precomposition in each coordinate acts on the
corresponding factor. Restricting this $k$-fold simplicial module along the diagonal $\alpha\mapsto(\alpha,\ldots,\alpha)$
gives $X$. For $k=2$, Quillen used this diagonal for bisimplicial groups \cite{Quillen1966}.

In degree $q$ the coordinate basis gives $\rank_A X_q=\prod_a(q+1+c_a)$.
The shape $(3,5)$ has simplicial dimension two. In that degree it gives
$\mathbf c=(0,2)$, so $\rank_A X_q=(q+1)(q+3)$
(\program{tex/anc/dstm_examples.py\#L25}{shape_example}).

A $3\times5$ matrix placed in degree one instead has $\mathbf c=(1,3)$.
The coordinate bases identify $X^{\mathbf c}(A)\otimes_A B$ with
$X^{\mathbf c}(B)$ for any ring homomorphism $A\to B$.

For $p\geq1$ and $0\leq i\leq p$, the index coface
$\Delta_i:I_{p-1}\to I_p$ skips $i$ in every coordinate.
For $p\geq0$ and $0\leq i\leq p$, the codegeneracy
$\Sigma_i:I_{p+1}\to I_p$ identifies $i$ and $i+1$ in every coordinate. Thus $d_iT=T\circ\Delta_i$ and
$s_iT=T\circ\Sigma_i$. For $m\in I_p$, write $\im(m)=\{m_1,\ldots,m_k\}$.
If $i\notin\im(m)$, then $m$ has a unique preimage $\Delta_i^{-1}m$. Hence
\begin{equation}\label{eq:face}
 d_iE_m=
 \begin{cases}
 0,&i\in\im(m),\\
 E_{\Delta_i^{-1}m},&i\notin\im(m).
 \end{cases}
\end{equation}
For any simplicial module $Y$, write $H_p(Y)$ for the homology of its
unaugmented chain complex, with differential
$\partial_p=\sum_{i=0}^p(-1)^i d_i$ for $p\geq1$ and $\partial_0=0$.
For the vector $v=(3,5)$ over $\mathbb Z$, $\partial_1 v=(5)-(3)=(2)$
(\program{tex/anc/dstm_examples.py\#L42}{vector_example}).

For $T\in X_n$, the simplicial submodule $\langle T\rangle$ generated by
$T$ has degree-$p$ term spanned by $X(\alpha)(T)$ for $\alpha:[p]\to[n]$.

\begin{example}\label{ex:matrix-cycle}
Over a field $K$, take the matrix unit $T=E_{01}$ in degree one of $X^{(0,0)}(K)$
and let $Y=\langle T\rangle$. Here $E_{ij}=E_{(i,j)}$. Both faces of $T$ vanish,
so $X(\alpha)(T)=0$ for every constant map $\alpha:[p]\to[1]$. Thus $Y_0=0$ and $Y_1=KT$.
The two surjections $[2]\to[1]$ give a spanning set $s_0T,s_1T$ for $Y_2$, with
\[
 T=\begin{pmatrix}0&1\\0&0\end{pmatrix},\qquad
 s_0T=E_{02}+E_{12},\quad s_1T=E_{01}+E_{02}.
\]
The simplicial identities give $\partial_2s_0T=s_0d_1T=0$ and $\partial_2s_1T=s_0d_0T=0$.
Thus $H_1(Y)\cong K$. In the ambient complex, $\partial_2E_{12}=T$.
Every matrix in $Y_2$ satisfies $x_{02}=x_{12}+x_{01}$, so $E_{12}\notin Y_2$.
See \program{tex/anc/dstm_examples.py\#L54}{generated_homology_example} for the calculation over $\mathbb Q$.
\end{example}

\section{Horn kernels and normalization}\label{sec:horn-kernels}
For a simplicial $A$-module $Y$, the normalized and degenerate chains in
degree $p$ are, respectively,
\[
 N_p(Y)=\bigcap_{i=1}^p\ker d_i,\qquad
 D_p(Y)=\sum_{i=0}^{p-1}\im s_i.
\]
Here $N_0(Y)=Y_0$ and $D_0(Y)=0$. The differential on $N(Y)$ is
$d_0$ in positive degrees and zero in degree zero
\cite[Exercise 8.3.4]{Weibel}. Reversing the face and degeneracy indices
in Dold's splitting \cite[Corollary 1.12]{Dold1958} gives
$Y_p=N_p(Y)\oplus D_p(Y)$. The inclusion of normalized chains
induces $H_p(N(Y))\cong H_p(Y)$ \cite[Theorem 8.3.8]{Weibel}.
Normalization gives the Dold--Kan equivalence with nonnegative chain
complexes \cite[Theorem 8.4.1]{Weibel}.

For $p\geq1$ and $0\leq j\leq p$, let $R_{p,j}(Y)=
\bigcap_{0\leq i\leq p,\ i\ne j}\ker d_i$.
The module $\Lambda_{p,j}(Y)$ of compatible horns consists of tuples
$(y_i)_{i\ne j}$ with $y_i\in Y_{p-1}$ satisfying $d_i y_\ell=d_{\ell-1}y_i$
for $i<\ell$, $i,\ell\ne j$. This condition is vacuous for $p=1$.
The horn map is
$\Phi_j:Y_p\to\Lambda_{p,j}(Y)$, $y\mapsto(d_i y)_{i\ne j}$.
Every simplicial module is Kan, so $\Phi_j$ is surjective
\cite[Lemma 8.2.8]{Weibel}. If $x$ fills a horn, its fillers form the
coset $x+R_{p,j}(Y)$.

Ashley's theorem on horn fillers in $D_p(Y)$ applies to the underlying
simplicial abelian group \cite[Ch.~3, Theorem~1.5]{Ashley1978}.
For simplicial vector spaces see Ronchi and Zhu \cite[Lemma~2.10]{RonchiZhu}.

\begin{lemma}\label{lem:hornsplit}
Let $Y$ be a simplicial $A$-module, $p\geq1$ and $0\leq j\leq p$.
Then $Y_p=R_{p,j}(Y)\oplus D_p(Y)$, and the horn map restricts to an
isomorphism $D_p(Y)\to\Lambda_{p,j}(Y)$.
\end{lemma}
\begin{proof}
For $0\leq j<p$, the simplicial identities give inverse maps
\[
 F_j=1-s_jd_j:R_{p,j}(Y)\longrightarrow R_{p,j+1}(Y),\qquad
 G_j=1-s_jd_{j+1}:R_{p,j+1}(Y)\longrightarrow R_{p,j}(Y).
\]
The faces of $s_jd_jx$ for $x\in R_{p,j}(Y)$ vanish away
from $j,j+1$. For $i<j$ use $d_id_j=d_{j-1}d_i$, and for
$i>j+1$ use $d_{i-1}d_j=d_jd_i$ after moving $d_i$ past $s_j$.
The identities $d_js_j=d_{j+1}s_j=1$ give $d_jF_jx=0$ and
$d_{j+1}F_jx=-d_jx$, so $G_jF_jx=x$. For $y\in R_{p,j+1}(Y)$ the
same face identities applied to $s_jd_{j+1}y$ give
$d_{j+1}G_jy=0$ and $d_jG_jy=-d_{j+1}y$, so $F_jG_jy=y$.
Both maps induce the identity modulo $D_p(Y)$. Since
$R_{p,0}(Y)=N_p(Y)$ maps isomorphically onto $Y_p/D_p(Y)$,
the same holds for every $R_{p,j}(Y)$. Since $\Phi_j$ is surjective with
kernel $R_{p,j}(Y)$, its restriction to $D_p(Y)$ is an isomorphism.
\end{proof}

\begin{theorem}\label{thm:hornbasis}
Let $X=X^{\mathbf c}(A)$, $p\geq1$, and $0\leq j\leq p$. Then
\begin{equation}\label{eq:hornbasis}
 R_{p,j}(X)=\bigoplus_{\substack{m\in I_p\\{}[p]\setminus\{j\}\subseteq\im(m)}} AE_m,
 \qquad
 \rank R_{p,j}(X)=
 \sum_{t=0}^p(-1)^t\binom pt\prod_{a=1}^k(p+c_a+1-t).
\end{equation}
Consequently, every horn has more than one filler for $p\leq k$ and a
unique filler for $p>k$.
\end{theorem}
\begin{proof}
By \eqref{eq:face}, $d_iT=0$ exactly when $T$ is supported on
tuples containing $i$. Intersecting these support conditions gives
the basis in \eqref{eq:hornbasis}. To count its elements by
inclusion--exclusion, forbid $t$ of the $p$ required labels.
This leaves $p+c_a+1-t$ choices in coordinate $a$.
A tuple of length $k$ cannot contain $p>k$ distinct values. If
$p\leq k$, assign the $p$ required values to distinct coordinates
and fill the remaining coordinates arbitrarily in $[p]$.
This gives a basis element since $A\ne0$.
\end{proof}

Each one-axis factor has normalized chains only in degrees zero and one.
The tensor product is a polynomial functor of degree $k$ on $k$-tuples of $A$-modules, so
$N_p(X)=0$ for $p>k$ also follows from the Dold--Puppe length bound
\cite[Hilfssatz~4.23]{DoldPuppe}. See also Satkurunath and K\"ock
\cite[Corollary~4.6]{SatkurunathKock}.

\begin{example}\label{ex:matrix-horn}
For shape $(3,3)$, take the $2$-horn with missing face $0$ and prescribed
faces $y_1,y_2$ below. They agree in their upper left entry, and their
fillers are the displayed matrices $T$ as $u,v$ range over $A$.
\[
 y_1=\begin{pmatrix}1&3\\7&9\end{pmatrix},\qquad
 y_2=\begin{pmatrix}1&2\\4&5\end{pmatrix},\qquad
 T=\begin{pmatrix}1&2&3\\4&5&u\\7&v&9\end{pmatrix}.
\]
The unique filler in $D_2(X)$ is $s_1y_2+s_0(y_1-y_2)$, which gives
$u=6$ and $v=8$. This matrix is a sum of degeneracies. Its adjacent columns
are distinct, so it is in neither $\im s_0$ nor $\im s_1$.
The matrix units $E_{(1,2)}$ and $E_{(2,1)}$
form a basis of $R_{2,0}(X)$. The function \program{tex/anc/dstm_examples.py\#L101}{horn_example}
computes the filler and verifies this kernel basis.
\end{example}

For $r\geq0$, a simplicial $A$-module is an $r$-hypergroupoid if its
underlying Kan complex has unique horn fillers in dimensions greater
than $r$ \cite{Duskin1979,Glenn1982}. By Theorem~\ref{thm:hornbasis},
the least such $r$ for $X$ is $k$.
For a shape of simplicial dimension $n$, the ambient simplicial module is an
$n$-hypergroupoid exactly when $k\leq n$.

\section{The normalized complex}\label{sec:normalized}
Put $C_a=\{1,\ldots,c_a\}$ and $B=A^{\prod_a C_a}$. Restriction to
the tuples $(p+t_1,\ldots,p+t_k)$ with $t_a\in C_a$ defines a simplicial
morphism $\epsilon:X\to\underline B$ to the constant module on $B$. Extension by zero gives a
section $\eta$. Let $Z_p$ denote the cycles of $N(X)$ augmented by
$\epsilon_0:X_0\to B$. In degree zero $Z_0=\ker\epsilon_0$ is spanned by
the $E_m$ with $0\in\im(m)$. In positive degrees cycles lie in the kernel of
every face. Hence \eqref{eq:face} gives
\begin{equation}\label{eq:cycles}
 Z_p=\bigoplus_{\substack{m\in I_p\\{}[p]\subseteq\im(m)}} AE_m
 \qquad(p\geq0).
\end{equation}

\begin{proposition}\label{prop:disks}
Let $X=X^{\mathbf c}(A)$. For $p\geq1$ and $0\leq j\leq p$, the sequence
\[
 0\longrightarrow Z_p\longrightarrow R_{p,j}(X)
 \xrightarrow{d_j} Z_{p-1}\longrightarrow0
\]
is split exact, with section $\iota_j$ of $d_j$ given by extension by
zero along $\Delta_j$.
Write $D^r(M)$ for the complex $M\xrightarrow{1}M$ in degrees $r,r-1$
and $B[0]$ for $B$ concentrated in degree zero.
The unaugmented normalized complex decomposes as
\begin{equation}\label{eq:disks}
 N_\bullet(X)\cong B[0]\oplus\bigoplus_{p=0}^{k-1}D^{p+1}(Z_p).
\end{equation}
In particular $H_0(X)\cong A^{\prod_a c_a}$ and $H_p(X)=0$ for $p>0$.
\end{proposition}
\begin{proof}
The $E_m$ in the basis of $R_{p,j}(X)$ with $j\in\im(m)$ span $Z_p$.
Deleting the gap at $j$ sends the remaining units bijectively to the basis of
$Z_{p-1}$ by \eqref{eq:face}, with inverse $\iota_j$.
For $j=0$ this gives $N_p(X)=Z_p\oplus\iota_0 Z_{p-1}$.
The differential is zero on $Z_p$ and inverts $\iota_0$ on the second summand.
In degree zero $X_0=\eta_0(B)\oplus Z_0$, and $Z_p=0$ for $p\geq k$
since a $k$-tuple cannot contain all of $[p]$.
\end{proof}

If $E_m$ is a coordinate unit of $N_p(X)$, then $m$ contains $1,\ldots,p$.
If $m$ also contains $0$, set
$h_pE_m=E_{m+\mathbf1}$, where $\mathbf1=(1,\ldots,1)$. Then
$\partial E_m=0$ and $\partial h_pE_m=E_m$.
If $m$ omits $0$, set $h_pE_m=0$. For $p>0$ we have
$\partial E_m=E_{m-\mathbf1}$, and $m-\mathbf1$ contains $0$ since $m$
contains $1$. Thus $h_{p-1}\partial E_m=E_m$.
In degree zero the units whose indices omit $0$ span $\eta_0(B)$,
on which both terms vanish. Hence $\partial h+h\partial=1-\eta\epsilon$.
If some $c_a=0$, then $B=0$ and $h$ contracts $N(X)$.
If every $c_a>0$, the retraction onto $B[0]$ gives $H_0(X)\cong B\ne0$.
The coordinate bases of $N_p(X)$, $Z_p$ and
$R_{p,j}(X)$ are indexed independently of $A$, so these modules commute
with arbitrary base change.

\begin{corollary}\label{cor:ranks}
Put $b_{-1}=0$ and let $b_p=\rank\im\partial_{p+1}$ in the normalized complex
for $p\geq0$. By Proposition~\ref{prop:disks}, $b_p=\rank Z_p$, with
$Z_0=\ker\epsilon_0$ in degree zero. Then
\[
 b_p=\sum_{t=0}^{p+1}(-1)^t\binom{p+1}{t}
          \prod_{a=1}^k(p+c_a+1-t),\qquad
 \rank N_p(X)=b_p+b_{p-1}+\mathbf 1_{p=0}\prod_a c_a.
\]
For $c_a=0$ on all axes, $b_p=(p+1)!\SII{k}{p+1}$, where
$\SII{k}{r}$ is a Stirling number of the second kind.
\end{corollary}
\begin{proof}
Excluding $t$ prescribed indices of $[p]$ leaves $p+1+c_a-t$ choices on
axis $a$. Inclusion--exclusion in \eqref{eq:cycles} gives $b_p$, and
\eqref{eq:disks} gives $\rank N_p(X)$. When every $c_a=0$, tuples covering $[p]$
are surjections from a $k$-element set onto $[p]$
\cite[\S1.4]{StanleyEC1}.
\end{proof}

\begin{example}
For matrices with $c_1=c_2=0$,
$N_\bullet(X)\cong D^1(A)\oplus D^2(A^2)$.
The normalized ranks are $1,3,2$ in degrees $0,1,2$. In degree one,
$d_0E_{(1,1)}=E_{(0,0)}$. In degree two, $d_0$ sends $E_{(1,2)}$
and $E_{(2,1)}$ to $E_{(0,1)}$ and $E_{(1,0)}$, respectively.
Over a field, $h_1E_{01}=E_{12}$ lies outside the submodule generated
by $E_{01}$ (Example~\ref{ex:matrix-cycle}).
The function \program{tex/anc/dstm_examples.py\#L164}{square_matrix_example} computes these ranks and images from the faces.
\end{example}

\begin{corollary}\label{cor:shape}
Over a fixed nonzero commutative ring $A$, $X^{\mathbf c}(A)\cong X^{\mathbf c'}(A)$
if and only if the entries of $\mathbf c$ and $\mathbf c'$ agree up to permutation.
Two shapes placed in their simplicial dimensions determine isomorphic
ambient modules if and only if one is obtained from the other by permuting
the axes and adding the same integer to every axis length.
\end{corollary}
\begin{proof}
Reducing an isomorphism modulo a maximal ideal of $A$ gives equal
dimensions in every degree. Hence the rank polynomials
$\prod_a(x+c_a+1)$ and $\prod_a(x+c'_a+1)$ agree at every nonnegative
integer and are equal in $\mathbb Z[x]$. Their roots determine the
multiset of integers $c_a$. Conversely, permuting the axes induces
the required simplicial isomorphism.
\end{proof}

\section{Chain contractions and ordinal sums}\label{sec:contractions}
The \v{C}ech nerve $V_A(E,W)$ of a split projection of $A$-modules
$E\oplus W\to W$ has degree-$p$ term $E^{p+1}\oplus W$
\cite[Tag 019P]{StacksCech}. Faces delete and degeneracies repeat
$E$-coordinates, fixing $W$. In the degreewise tensor product
$V=\bigotimes_a V_A(E_a,W_a)$ over $A$, the projections and zero sections give
$\epsilon:V\to\underline B$ and $\eta:\underline B\to V$, where
$B=\bigotimes_a W_a$. With $E_a=A$ and $W_a=A^{c_a}$ the coordinate
bases identify $V$ with $X^{\mathbf c}(A)$.

Let $H_p:V_p\to V_{p+1}$ insert a zero before the first $E_a$-coordinate in each factor.
With $H_{-1}=\eta_0$ and $d_0=\epsilon_0$ on $V_0$, deleting the
inserted coordinate gives $d_0H_p=1$. Deleting any other gives
$d_iH_p=H_{p-1}d_{i-1}$ for $1\leq i\leq p+1$.
Taking alternating sums gives $\partial H+H\partial=1$ in positive degrees
and $\partial_1H_0=1-\eta_0\epsilon_0$ in degree zero. Thus the unaugmented
complex retracts by chain homotopy onto $B[0]$, with retraction $\epsilon_0$
and section $\eta_0$ in degree zero and both maps zero in higher degrees.
Since $H_ps_i=s_{i+1}H_{p-1}$
for $0\leq i<p$, $H_p$ preserves degenerate chains. The induced homotopy on
$N_p(V)$ is $H_p|_{N_p(V)}$ followed by projection along $D_{p+1}(V)$.
For $X^{\mathbf c}(A)$ with some $c_a=0$, we have $B=0$ and $H$ is a
contraction. For square matrices \program{tex/anc/dstm_examples.py\#L164}{square_matrix_example}
checks the identity on every matrix unit in degrees zero through two.

For $X^{\mathbf c}(A)$, $H_pE_m=E_{m+\mathbf1}$. If $E_m\in N_p(X)$,
its image is normalized when $0\in\im(m)$ and equals $s_0E_m$
otherwise. The projection thus gives the homotopy of
Section~\ref{sec:normalized}. Axis permutations preserving $\mathbf c$
commute with both homotopies. Hence the homotopies preserve the fixed
submodule and the submodule on which these permutations act by their signs.
On square complex matrices they commute with conjugate transpose over $\mathbb R$.

Suppose some $c_a=0$. For $t\in\mathbb Z$ let $F^tX_p$ be the span
of the $E_m$ with $\min_a m_a\geq p+t$. A nonzero face lowers the
minimum by at most one, and $H_p$ raises it by one. Thus $\partial$
and $H$ preserve this decreasing filtration, finite in each degree,
and $H$ contracts each quotient $F^tX_\bullet/F^{t+1}X_\bullet$.
The homology groups of these quotients form the $E^1$-page of the
associated spectral sequence, so $E^1=0$.

For arbitrary excesses, write $X^+$ for $X$ augmented by $\epsilon_0$ and
$\Delta_+$ for $\Delta$ with the empty ordinal $[-1]$ adjoined. For $p,q\geq-1$ put
\[
 I_{p,q}=\prod_a\bigl([p]\oplus[c_a-1]\oplus[q]\bigr),\qquad
 Y_{p,q}=A^{I_{p,q}}.
\]
Precomposition with $\prod_a(\alpha\oplus1_{[c_a-1]}\oplus\beta)$ gives
$Y(\alpha,\beta)$ for morphisms $\alpha,\beta$ in $\Delta_+$. Exchanging
the last two summands as sets identifies $Y_{p,q}$ with $X^+_{p+q+1}$ and
$Y(\alpha,\beta)$ with $X^+(\alpha\oplus\beta)$. In nonnegative bidegrees this is the
total d\'ecalage of $X$ \cite[Lemma~3.1]{StevensonDecalage}.
At fixed $q$, the horizontal row has excess $c_a+q+1$ on axis $a$.
For the unaugmented row, Proposition~\ref{prop:disks} gives $H_p=0$ for
$p>0$, and the augmentation identifies $H_0$ with
$Y_{-1,q}\cong A^{\prod_a(c_a+q+1)}$.
For $\mathbf c=(0,0)$ the vertical faces on $Y_{-1,9}$ are the ten
principal deletions used in the Petersen computation
(\program{experiments/petersen_brd_right_cosets.py\#L105}{vertical_face_rows}).

For $n\geq-2$, take $\mathrm{Tot}_n Y=\bigoplus_{p+q=n,\;p,q\geq-1}Y_{p,q}$ with
both augmentations and with each differential zero out of index $-1$.
Extension by zero along the first horizontal coface gives
$H_{p,q}:Y_{p,q}\to Y_{p+1,q}$ and
$\partial_hH+H\partial_h=1$, including at $p=-1$.
The vertical faces commute with $H$. Since $H$ raises $p$ by one, the
vertical terms for $\partial=\partial_h+(-1)^p\partial_v$ cancel in
$\partial H+H\partial$, giving the identity on $\mathrm{Tot}_\bullet Y$.

\section{Generated submodules}
\label{sec:generated}
Fix a field $K$, $n\geq0$ and $\mathbf c\in\mathbb Z_{\geq0}^k$.
Put $X=X^{\mathbf c}(K)$ and let $P=K[\Delta^n]$ be the free simplicial
$K$-module on the standard $n$-simplex. A basis of $P_p$ consists of
the nondecreasing maps $[p]\to[n]$.
For $T\in X_n$, the representing property of $P$ gives
\[
 f_T:P_\bullet\longrightarrow X_\bullet,
 \qquad f_{T,p}(\alpha)=X(\alpha)(T)
 \quad\bigl(\alpha:[p]\to[n]\bigr).
\]
Its image is the simplicial submodule $\langle T\rangle$ generated by $T$.
For any simplicial module $Y$, the projection $\nu_p:Y_p\to N_p(Y)$ is
\[
 \nu_0=1,\qquad
 \nu_p=(1-s_0d_1)\cdots(1-s_{p-1}d_p)\quad(p>0).
\]
Apply the factors from right to left. Each $1-s_{i-1}d_i$ has image in
$\ker d_i$, preserves $\ker d_j$ for $j>i$, and induces the identity
modulo $D_p(Y)$. The splitting $Y_p=N_p(Y)\oplus D_p(Y)$ gives the formula.
In $P_p$ the noninjective maps span $D_p(P)$, so the $\nu_p\alpha$
for increasing injections $\alpha:[p]\hookrightarrow[n]$ form a basis of $N_p(P)$.
Since $f_T$ commutes with faces and degeneracies, it sends this basis
to the columns $\nu_p X(\alpha)(T)$ spanning $N_p(\langle T\rangle)$.
Let $r_p(T)$ be their rank. For $p>0$ let $\delta_p(T)$ be the rank of their
images under $d_0$, and put $\delta_0(T)=0$.
The $p$th Betti number is $r_p(T)-\delta_p(T)-\delta_{p+1}(T)$.
Over $\mathbb Q$ or a finite prime field,
\program{src/simplicial_tensors/homology.py\#L132}{generated_homology} takes a hypermatrix in the simplicial dimension of its shape.
For a tensor with specified degree and excesses, use
\program{src/simplicial_tensors/tensor_chain_complex.py\#L130}{sparse_generated_homology}.
For $p>n$ there are no columns. For $p>k$, $N_p(X)=0$ by
Theorem~\ref{thm:hornbasis}.

Put
\[
 C=N(P),\qquad g_T=N(f_T),\qquad L(T)=\ker g_T.
\]
Normalization is exact, so
\begin{equation}\label{eq:generated-normalization}
 N(\langle T\rangle)\cong C/L(T).
\end{equation}
Here $\dim C_p=\binom{n+1}{p+1}$ for $0\leq p\leq n$ and $C_p=0$ for $p>n$.
By Dold--Kan the subspaces $L_p(T)$ determine $\ker f_T$. For
$T,T'\in X_n$, $L(T)=L(T')$ holds precisely when there is a simplicial
isomorphism $\langle T\rangle\to\langle T'\rangle$ sending $T$ to $T'$.
The matrix units $E_{01},E_{10}\in X^{(0,0)}_1(K)$ have the same
kernel ($L_0=C_0$, $L_1=0$) and span distinct lines in degree one
(\program{tex/anc/dstm_examples.py\#L54}{generated_homology_example}).

Since $H_0(C)=K$ and
$H_j(C)=0$ for $j>0$, the short exact sequence
$0\to L(T)\to C\to C/L(T)\to0$ gives
\begin{equation}\label{eq:generated-homology}
 \begin{gathered}
 H_j(\langle T\rangle)\cong H_{j-1}(L(T))\qquad(j\geq2),\\
 0\longrightarrow H_1(\langle T\rangle)\longrightarrow H_0(L(T))
 \xrightarrow{\epsilon}K\longrightarrow H_0(\langle T\rangle)
 \longrightarrow0,
 \end{gathered}
\end{equation}
where $\epsilon$ is induced by the augmentation sending every vertex of
$\Delta^n$ to $1$.

\subsection{Quotients and automorphisms}

To classify the generated modules after forgetting their generators, it suffices
by Dold--Kan to classify the quotients $C/L(T)$ up to chain isomorphism.
Write $\operatorname{Ch}_{\geq0}(K)$ for the category of chain complexes of
$K$-vector spaces concentrated in nonnegative degrees. A complex is finite
if the direct sum of its chain groups is finite-dimensional. For a complex $E$,
write
\[
 Z_p(E)=\ker(\partial_p:E_p\to E_{p-1}),\qquad
 B_p(E)=\im(\partial_{p+1}:E_{p+1}\to E_p),
 \qquad E_{-1}=0.
\]
Thus $Z_0(E)=E_0$. For a $K$-vector space $V$ and
$p\geq0$, let $S^p(V)$ be the complex with $V$ in degree $p$ and zero elsewhere. Choosing complements of boundaries
inside cycles, and of cycles inside each chain group, gives
\begin{equation}\label{eq:finite-chain-splitting}
 E\cong\bigoplus_{p\geq0}S^p(H_p(E))
       \oplus\bigoplus_{p\geq1}D^p(B_{p-1}(E))
\end{equation}
for every finite complex $E$. The differential maps the chosen
complement of $Z_p(E)$ isomorphically onto $B_{p-1}(E)$.
Thus the homology dimensions and differential ranks determine the
isomorphism class of a finite complex. Their additivity gives cancellation
under direct sums.

For the normalized simplex complex, the dimensions of its chain groups
and its homology give
\begin{equation}\label{eq:simplex-projective}
 C\cong S^0(K)\oplus\bigoplus_{p=1}^n
                   D^p\bigl(K^{\binom np}\bigr).
\end{equation}
Chain maps from $S^0(V)$ and $D^p(V)$ to $E$ correspond to linear
maps $V\to E_0$ and $V\to E_p$, respectively. Evaluation in each
degree and $\operatorname{Hom}_K(V,-)$ are exact, so these summands
are projective and hence so is $C$.

\begin{proposition}\label{prop:quotient-lift}
Fix a field $K$ and $n\geq0$, and let $C=N(K[\Delta^n])$.
For subcomplexes $L,M\subseteq C$, every chain isomorphism
$\varphi:C/L\to C/M$ lifts to a chain automorphism $u:C\to C$
such that the square commutes.
\[
\begin{tikzcd}
 C \arrow[r,"u"] \arrow[d,"\pi_L"'] & C \arrow[d,"\pi_M"]\\
 C/L \arrow[r,"\varphi"'] & C/M .
\end{tikzcd}
\]
Here $\pi_L$ and $\pi_M$ are the quotient maps. In particular,
$C/L\cong C/M$ if and only if $u(L)=M$ for some
$u\in\operatorname{Aut}_{\operatorname{Ch}_{\geq0}(K)}(C)$.
\end{proposition}
\begin{proof}
Projectivity supplies chain maps $f,g:C\to C$ with
$\pi_Mf=\varphi\pi_L$ and $\pi_Lg=\varphi^{-1}\pi_M$.
Put $h=gf$ and $h'=fg$. For sufficiently large $m$ their Fitting
decompositions give direct sums of subcomplexes
\[
 C=U\oplus V=\im h^m\oplus\ker h^m,\qquad
 C=U'\oplus V'=\im (h')^m\oplus\ker (h')^m.
\]
The identities $fh=h'f$ and $gh'=hg$ give $f(U)\subseteq U'$ and
$g(U')\subseteq U$. Since $gf$ is invertible on $U$, $f|_U$ is
injective. Since $U'=fg(U')\subseteq f(U)$, it maps onto $U'$.
Cancellation in \eqref{eq:finite-chain-splitting} gives
an isomorphism $a:V\to V'$. From $\pi_L(gf)^m=\pi_L$ and
$\pi_M(fg)^m=\pi_M$ we get $V\subseteq L$ and $V'\subseteq M$.
The automorphism $u=f|_U\oplus a$ then satisfies $\pi_Mu=\varphi\pi_L$
on $U$, by the choice of $f$, and on $V$, where both sides vanish.
Taking kernels gives $u(L)=M$.
Conversely, such an automorphism induces an isomorphism of quotients.
\end{proof}

Complexes in degrees $0,\ldots,n$ are modules over the path algebra
of $n\to\cdots\to0$ modulo paths of length two. This algebra and its opposite are
finite-dimensional and hence semilocal of stable range one
\cite[Corollary~6.5]{Bass1964}. Applied to $\varphi\pi_L$ and $\pi_M$,
both with finitely generated projective source $C$, Canfell's theorem on right
equivalence of presentations \cite[p.~141]{Canfell1992} also gives the lift.

For $a\in\operatorname{Aut}_{\operatorname{Ch}_{\geq0}(K)}(C)$,
let $\widetilde a:P\to P$ be its simplicial automorphism under Dold--Kan.
The tensor
\[
 T'=f_{T,n}\bigl(\widetilde a_n(\operatorname{id}_{[n]})\bigr)\in X_n
\]
satisfies $f_{T'}=f_T\widetilde a$ by the representing property of $P$.
Thus $\langle T'\rangle=\langle T\rangle$, and $L(T)\cdot a=L(T')=a^{-1}L(T)$
defines a right action on $\{L(T):T\in X_n\}$. By Proposition~\ref{prop:quotient-lift},
its orbits classify the generated modules up to simplicial isomorphism.
By \eqref{eq:finite-chain-splitting}, the chain dimensions $r_p(T)$ and
differential ranks $\delta_p(T)$ determine this isomorphism class.
Since $C$ is finite and $\dim C_p$ bounds both ranks, there are finitely many orbits.
The \texttt{lifting} case of \program{experiments/kernel_sequences.py}{kernel_sequences.py}
constructs a lift and checks the commutative square in each degree.

\subsection{Realizable kernels}

Put $\mathcal D=N(X)$. A subcomplex $L\subseteq C$ is
\emph{realizable in $X$} if $L=L(T)$ for some $T\in X_n$.

\begin{proposition}\label{prop:kernel-realizability}
Let $K$ be a field, $n\geq0$, $C=N(K[\Delta^n])$, and
$X=X^{\mathbf c}(K)$ with $\mathbf c\in\mathbb Z_{\geq0}^k$.
For a subcomplex $L\subseteq C$ and $Q=C/L$, the following are equivalent:
\begin{enumerate}
\item $L$ is realizable in $X$.
\item There is a chain monomorphism $Q\hookrightarrow\mathcal D=N(X)$.
\item For every $p\geq0$,
\begin{equation}\label{eq:realizability-inequalities}
 \dim_K Z_p(Q)\leq\dim_K Z_p(\mathcal D),\qquad
 \dim_K B_p(Q)\leq\dim_K B_p(\mathcal D).
\end{equation}
\end{enumerate}
\end{proposition}
\begin{proof}
If $L=L(T)$, then $g_T$ induces an injection $Q\hookrightarrow\mathcal D$.
Such an injection restricts to injections on cycles and boundaries.

Conversely, the boundary inequality gives an injection
$B_p(Q)\hookrightarrow B_p(\mathcal D)$ in each degree. By the cycle
inequality, it extends to an injection $j_p:Z_p(Q)\hookrightarrow Z_p(\mathcal D)$.
Choose a complement $W_p$ of $Z_p(Q)$ in $Q_p$. For $p>0$,
$\partial_Q:W_p\to B_{p-1}(Q)$ is an isomorphism. Lift
$j_{p-1}\partial_Q|_{W_p}$ through
$\partial_{\mathcal D}:\mathcal D_p\to B_{p-1}(\mathcal D)$ to a
linear map $\ell_p:W_p\to\mathcal D_p$. Set
$i_p(z+w)=j_p(z)+\ell_p(w)$ and take $i_0=j_0$, since $W_0=0$.
For $p>0$ we have $\partial_{\mathcal D}i_p(z+w)=j_{p-1}\partial_Qw
=i_{p-1}\partial_Q(z+w)$, so $i$ is a chain map. If $i_p(z+w)=0$,
this identity gives $w=0$ and then $z=0$. Since $i_0=j_0$ is injective,
$i:Q\hookrightarrow\mathcal D$ is a chain monomorphism.

Under Dold--Kan the composite $C\to Q\xrightarrow{i}\mathcal D$
corresponds to $f_T:P\to X$, where $T$ is the image of the identity
$n$-simplex. The kernel of the composite is $L$, so $L=L(T)$.
\end{proof}

For $p>0$, the cycles and boundaries of $\mathcal D$ coincide and have
dimension $b_p$ by Corollary~\ref{cor:ranks}. In degree zero,
\[
 \dim_K Z_0(\mathcal D)=\prod_a(c_a+1),\qquad
 \dim_K B_0(\mathcal D)=\prod_a(c_a+1)-\prod_a c_a.
\]
Since $B_p(Q)\subseteq Z_p(Q)$, the boundary inequality adds a condition
only in degree zero and only if every $c_a>0$. In particular, the cycle
inequalities suffice when the excesses are computed from a shape in its simplicial dimension.

For example, when $\mathbf c=(1,1)$, the cycle dimensions
of $\mathcal D$ are $(4,2,0)$ and its boundary dimensions are $(3,2,0)$
(\program{tex/anc/dstm_examples.py\#L195}{shifted_matrix_example}).
The complex $Q=D^1(K^4)$ is a quotient of $N(K[\Delta^4])$ by
\eqref{eq:simplex-projective}. It satisfies the cycle inequalities,
but $\dim B_0(Q)=4>3=\dim B_0(\mathcal D)$, so it does not embed.
The \texttt{obstruction} case of \program{experiments/kernel_sequences.py}{kernel_sequences.py}
checks this inequality. The \texttt{realize} case constructs a tensor from a prescribed kernel.

\subsection{Matrix examples and homology in higher degrees}

\begin{example}\label{ex:parameter-jump}
For $\mathbf c=(0,0)$ and $t\in K$, let
\[
 T(t)=\begin{pmatrix}0&1&0\\0&1&1\\1&1&t\end{pmatrix}\in X_2,
 \qquad Y_t=\langle T(t)\rangle.
\]
Use the ordered basis $(E_{01},E_{10},E_{11})$ of $N_1(X)$.
For the restriction $y=\left(\begin{smallmatrix}1&1\\1&t\end{smallmatrix}\right)$
to vertices $1,2$, we have $\nu_1y=y-s_0(1)=(t-1)E_{11}$.
The three normalized restrictions are
\[
 v_{01}=(1,0,1),\qquad v_{02}=(0,1,t),\qquad
 v_{12}=(0,0,t-1).
\]
The matrix with these columns has determinant $t-1$. In degree two,
$\nu_2T(t)=E_{12}-E_{21}$ and $d_0\nu_2T(t)=E_{01}-E_{10}$.
The differential $N_1(X)\to K$ sends $(a,b,c)$ to $c$.
Thus the differential ranks of the generated complex are $(0,1,1)$,
indexed by source degree, and
\[
\begin{array}{c|c|c}
 &(\dim N_0(Y_t),\dim N_1(Y_t),\dim N_2(Y_t))&(\dim H_0(Y_t),\dim H_1(Y_t),\dim H_2(Y_t))\\ \hline
 t\ne1&(1,3,1)&(0,1,0)\\
 t=1&(1,2,1)&(0,0,0).
\end{array}
\]
For $t\ne1$, the class of $E_{01}=v_{01}-(t-1)^{-1}v_{12}$
generates $H_1(Y_t)$.
At $t=1$, the cycle space in degree one is the line spanned by
$v_{01}-v_{02}=E_{01}-E_{10}$, which is already a boundary.

In the vertex basis of $C_0=K^3$, the degree-zero kernel is
\[
 L_0(T(t))=\ker(0,1,t)
   =\operatorname{span}_K\{(1,0,0),(0,-t,1)\}.
\]
These planes vary with $t$. For $t\ne1$, $L_1=L_2=0$, and the quotient
complexes have the same chain dimensions and differential ranks. Hence their kernels
lie in one orbit by Proposition~\ref{prop:quotient-lift}. At $t=1$,
$L_1$ is spanned by $\alpha-s_0d_1\alpha$, where $\alpha:[1]\to[2]$
is the increasing injection with image $\{1,2\}$, and $L_2=0$.

For the integer matrix $T(4)$, the submodule generated over $\mathbb Q$ has
first Betti number one. Reducing the entries modulo $3$ gives $T(1)$ over
$\mathbb F_3$, whose generated submodule has first Betti number zero.
Over $\mathbb Z$ the three normalized columns form a basis of the lattice
they span, of index three in $N_1(X^{(0,0)}(\mathbb Z))$. In this basis
the differentials in degrees one and two are $(1,4,3)$ and
$(1,-1,1)^{\mathsf T}$, respectively. Both have rank one modulo $3$,
so tensoring the integral generated complex with $\mathbb F_3$ gives
first Betti number one.
The \texttt{jump} case of \program{experiments/kernel_sequences.py}{kernel_sequences.py}
computes the ranks at integral values of $t$ over either field.
\end{example}

\begin{example}[The simplicial sphere]\label{ex:sparse-higher-cycle}
Fix $q\geq1$, $k=q+1$ and $\mathbf c\in\mathbb Z_{\geq0}^k$,
and take $T_q=E_{(0,1,\ldots,q)}\in X^{\mathbf c}_q(K)$.
Every face of $T_q$ is zero by \eqref{eq:face}. Sending the top simplex to
$T_q$ therefore induces a surjection $K[\Delta^q]/K[\partial\Delta^q]\to
\langle T_q\rangle$. The source is the reduced free simplicial $K$-module on
the sphere $\Delta^q/\partial\Delta^q$. On normalized chains the surjection sends the sole generator,
in degree $q$, to the nonzero cycle $T_q$. It is therefore an isomorphism,
and Dold--Kan gives an isomorphism of simplicial modules. In particular,
\[
 N(\langle T_q\rangle)\cong S^q(K),\qquad
 H_q(\langle T_q\rangle)\cong K,\qquad
 H_p(\langle T_q\rangle)=0\quad(p\ne q).
\]
In the full module,
\[
 \partial_{q+1}E_{(1,2,\ldots,q+1)}=T_q.
\]
For $k\leq q$ no index tuple contains $[q]$, so \eqref{eq:cycles} gives
$Z_q(N(X))=0$. Hence $q+1$ is the least number of axes on which a simplicial
submodule can have nonzero $q$th homology.
\end{example}

For $\mathbf c=0$, the \texttt{spheres} case of
\program{experiments/kernel_sequences.py}{kernel_sequences.py} computes the
reduced Betti numbers. For the $2$-sphere it gives $(0,0,1)$ in degrees $0,1,2$.

\subsection{Encoding graph homology by tensors}

The $v$ vertices and $e$ edges of a finite graph $G$ are its $0$-cells
and $1$-cells. Its cellular chain complex over $A$ is the incidence complex
$A^e\to A^v$, so its homology is the ordinary graph homology $H_*(G;A)$.
Here $A$ is a nonzero commutative ring with identity, and $G$ may have
loops and parallel edges. Number the vertices and edges separately from $1$
and orient edge $j$ from $u_j$ to $w_j$.
Take $X=X^{(r,r,r)}(A)$ with $r\geq\max(v,e)$. Assign to each
vertex $t$ and each edge $j$ the tensors
\[
 a_t=E_{(0,t,t)}\in X_0,\qquad
 T_j=E_{(0,u_j+1,u_j+1)}+E_{(1,w_j+1,w_j+1)}
              +E_{(0,1,j+1)}\in X_1.
\]
Deleting index $1$ gives $d_1T_j=a_{u_j}$, and deleting index $0$ gives
$d_0T_j=a_{w_j}$.

\begin{proposition}\label{prop:graph}
Let $Y_G\subseteq X^{(r,r,r)}(A)$ be the simplicial submodule generated by
the vertex tensors $a_t$ and edge tensors $T_j$. The assignments
$t\mapsto a_t$ and $j\mapsto b_j=T_j-s_0a_{u_j}$ give an isomorphism
from the incidence complex of $G$ over $A$ to $N(Y_G)$.
Hence $H_*(Y_G)\cong H_*(G;A)$.
\end{proposition}
\begin{proof}
The $a_t$ are distinct basis vectors of $X_0$ and span $(Y_G)_0$.
The normalized edges $b_j$ satisfy $d_0b_j=a_{w_j}-a_{u_j}$ and are linearly
independent because at $(0,1,j+1)$ only $b_j$ has a nonzero entry, equal to $1$.
The $T_j$ and $s_0a_t$ span $(Y_G)_1$. Their projections are $b_j$ and zero,
respectively, so the $b_j$ span $N_1(Y_G)$. For $p>1$,
every map $[p]\to[0]$ or $[p]\to[1]$ factors through a codegeneracy,
so $N_p(Y_G)=0$.
\end{proof}

The generators in degrees $p=0,1$ use only first-axis indices in $[p]$.
Extension by zero on the additional first-axis indices commutes with simplicial
operators. Thus $Y_G$ lies in $X^{(0,r,r)}(A)$, of shape $(1,r+1,r+1)$ in degree
zero. The ambient complex $N(X^{(0,r,r)}(A))$ is contractible by Proposition~\ref{prop:disks}.

On each component, the boundaries of the edges of a spanning tree form
a basis of the kernel of the augmentation $a_t\mapsto 1$.
From each remaining edge subtract the signed sum of the edges on the tree
path from its source to its target, with signs given by the orientations.
These differences form a basis of the cycle module.
Thus $H_0(Y_G)\cong A^b$ and $H_1(Y_G)\cong A^{e-v+b}$, where $b$
is the number of components. The Petersen graph has ten vertices,
fifteen edges, and one component, so its first homology is $A^6$.
Over $\mathbb Q$,
\program{tex/anc/rank_shifted_homology.py}{rank_shifted_homology.py}
computes rank nine for the incidence map from the normalized tensor faces.

For a two-axis module $X'=X^{(c_1,c_2)}(A)$, the cycles in $N_1(X')$ have basis $E_{(0,1)}$ and
$E_{(1,0)}$ by \eqref{eq:cycles}. Over a field, every simplicial
submodule of $X'$ therefore has first Betti number at most two. In particular,
the Petersen incidence complex cannot occur on two axes over a field.

\subsection{Boundary ranks of ordered adjacency matrices}

For a vertex order $\pi$ of a finite simple graph $H$, let $C_\pi$ be its
adjacency matrix in that order. We compute $\rank_{\mathbb F_2}\partial C_\pi$
for two Cayley graphs on $V=(\mathbb Z/4\mathbb Z)^2$. The connection
set for the $4\times4$ rook graph consists of the six nonzero multiples of
the standard basis vectors $e_1,e_2$. For the Shrikhande graph it is
$\{\pm e_1,\pm e_2,\pm(e_1+e_2)\}$. Apply each affine bijection
$x\mapsto Mx+t$ of $V$ to the lexicographic order on $\{0,1,2,3\}^2$.
There are $96$ choices of $M$ and $16$ translations.
The program \program{experiments/affine_boundary_ranks.py}{affine_boundary_ranks.py}
computes Table~\ref{tab:affine-ranks} by summing the principal deletion faces of $C_\pi$.

\begin{table}[htb]
\centering\small
\caption{Boundary ranks of the original adjacency matrices over $\mathbb F_2$.
Each row counts the $1{,}536$ affine orders of the specified Cayley graph.}
\label{tab:affine-ranks}
\begin{tabular}{l|rrrr}
Rank of $\partial C_\pi$ & 8 & 10 & 12 & 14 \\ \hline
$4\times4$ rook & 512 & 256 & 512 & 256 \\
Shrikhande & 1,536 & 0 & 0 & 0
\end{tabular}
\end{table}

For the Petersen graph, over all $10!$ orders, the ranks $4,6,8$ occur
$45{,}120$, $1{,}019{,}520$, and $2{,}564{,}160$ times, respectively
(\program{experiments/petersen_brd_right_cosets.py}{petersen_brd_right_cosets.py}).

For both $K_4$ and its edgeless complement, the four deletion faces coincide,
so $\partial C_\pi=0$ for every order. Their numbers of oriented Hamiltonian
paths are $24$ and $0$, respectively
(\program{experiments/interleaved_graph_paths.py\#L94}{small_examples}).
Thus the boundary ranks of the original matrices do not determine the path count.
Insert one isolated vertex between successive vertices of $\pi$, retaining
the edges of $H$, and let $A_\pi$ be the adjacency matrix in this order.
Figure~\ref{fig:interleaved-path} shows the construction for $P_3$.

\begin{figure}[htb]
\centering
\begin{tikzpicture}[
  font=\small,
  vertex/.style={circle,draw,fill=black,inner sep=0pt,minimum size=4pt},
  isolated/.style={vertex,fill=white},
  every label/.style={font=\small,inner sep=1pt},
  edge line/.style={semithick}]
  \begin{scope}
    \node at (1,0.8) {Path $P_3$};
    \draw[edge line] (0,0) -- (2,0);
    \foreach \i in {0,1,2}
      \node[vertex,label=below:$v_{\i}$] at (\i,0) {};
  \end{scope}
  \begin{scope}[xshift=3.8cm,x=0.8cm]
    \node at (2,0.8) {Interleaved order};
    \draw[edge line] (0,0) to[bend left=34] (2,0)
                           to[bend left=34] (4,0);
    \foreach \i in {0,1,2}
      \node[vertex,label=below:$v_{\i}$] at (2*\i,0) {};
    \foreach \i in {1,2}
      \node[isolated,label=below:$z_{\i}$] at (2*\i-1,0) {};
  \end{scope}
  \begin{scope}[xshift=9cm,x=0.73cm]
    \node at (1.7,0.8) {Graph with adjacency matrix $\partial A_\pi$};
    \draw[edge line] (0,0) -- (1,0) (2.4,0) -- (3.4,0);
    \foreach \i/\x in {0/0,1/1,2/2.4,3/3.4}
      \node[vertex,label=below:$\i$] at (\x,0) {};
  \end{scope}
\end{tikzpicture}
\caption{The open circles are the inserted isolated vertices. For
$\pi=(v_0,v_1,v_2)$ the boundary is the adjacency matrix of two disjoint
edges, so $\rank(\partial A_\pi)=4$ over $\mathbb F_2$
(\program{experiments/interleaved_graph_paths.py\#L94}{small_examples}).}
\label{fig:interleaved-path}
\end{figure}

\begin{samepage}
\begin{proposition}[Hamiltonian paths and the mod-$2$ boundary]\label{prop:ordered-graph-ranks}
Let $H$ be a finite simple graph on $n\geq2$ vertices, with adjacency entries
$a_{uv}$. For a vertex order $\pi=(v_0,\ldots,v_{n-1})$, adjoin $n-1$ isolated
vertices to $H$ and let $A_\pi$ be the adjacency matrix in the order
$v_0,z_1,v_1,\ldots,z_{n-1},v_{n-1}$. Let $N_\pi$ be the adjacency matrix of $H$
with its vertices ordered by $\pi$ and its edges oriented from earlier to later.
Deleting each row and its matching column in turn and summing over $\mathbb F_2$
gives a boundary matrix $\partial A_\pi$ satisfying
\[
 \rank(\partial A_\pi)=2\rank N_\pi,\qquad
 \det(\partial A_\pi)=\prod_{i=0}^{n-2}a_{v_i v_{i+1}}.
\]
Thus $\partial A_\pi$ has full rank $2n-2$ exactly when $\pi$ traces a
Hamiltonian path of $H$. The number of such orders is twice the number
of Hamiltonian paths counted up to reversal.

Let $p(H)$ be the minimum number of vertex-disjoint paths covering $V(H)$.
Single vertices are allowed, and the paths need not be induced.
Then $\max_\pi\rank(\partial A_\pi)=2(n-p(H))$.
\end{proposition}
\end{samepage}
\begin{proof}
For an edge $v_iv_j$ with $i<j$, deleting either
endpoint removes the edge. There are $2i$ deletions before both endpoints,
$2(j-i)-1$ between them, and $2(n-1-j)$ after both.
Only the deletions between the endpoints contribute modulo $2$,
at $(2i,2j-1)$ and its transpose.
Grouping even indices before odd ones gives off-diagonal blocks $U,U^\top$
and zero diagonal blocks, where $U$ is $N_\pi$ without its zero first column
and zero last row. Hence $\rank(\partial A_\pi)=2\rank N_\pi$.
The upper triangular matrix $U$ has diagonal entries $a_{v_i v_{i+1}}$,
so $\det(\partial A_\pi)=(\det U)^2$ is their product in $\mathbb F_2$.

Expanding a nonzero minor of $N_\pi$ of order $r$ gives a choice of
$r$ edges with distinct initial vertices and distinct terminal
vertices. All edges point forward, so they give a cover by $n-r$ paths,
including singletons. Hence $\rank N_\pi\leq n-p(H)$. Conversely,
concatenate the paths of a minimum cover, listing each in traversal order.
Their $n-p(H)$ edges join consecutive vertices in this order. Taking
their initial vertices as rows and their terminal vertices as columns, in
order, gives an upper triangular submatrix with diagonal entries $1$,
so this order attains the bound.
\end{proof}

Let $J$ be the all-ones matrix. By Stanley's formula
\cite[Corollary, p.~171]{StanleyPaths}, the coefficient of $t^j$ in
$\det(I+t(J-N_\pi))$ over $\mathbb Z$ counts the directed paths on $j$ vertices
in the forward orientation, $1\leq j\leq n$. For $j=n$ only $\pi$ can be such a path, so
$\det(J-N_\pi)\bmod 2=\det(\partial A_\pi)$. Subtracting the last row of $J-N_\pi$
from each earlier row gives $\rank(J-N_\pi)=1+\rank N_\pi$ over
$\mathbb F_2$, since the first column of $N_\pi$ is zero.

For the $4\times4$ rook graph, counting full-rank mod-$2$ boundaries
recovers the number of oriented Hamiltonian paths. Summing the $0$--$1$
values of $\det(\partial A_\pi)$ as integers over all vertex orders gives
\[
 \#\{\pi:\rank_{\mathbb F_2}(\partial A_\pi)=30\}
 =\sum_\pi\det(\partial A_\pi)
 =\sum_\pi\prod_{i=0}^{14}a_{v_i v_{i+1}}=22{,}394{,}880.
\]
This agrees with the fourth term of OEIS A096970 \cite{OEISRook},
which counts each path and its reverse separately. The program
\program{experiments/interleaved_graph_paths.py}{interleaved_graph_paths.py}
evaluates the sum by grouping paths by endpoint and set of unused vertices.

Table~\ref{tab:interleaved-ranks} gives the distribution of
$\rank(\partial A_\pi)$ over all orders of the original vertices. An
isomorphism of graphs sends each order to an order with the same ordered
adjacency matrix after insertion, so this distribution is a graph invariant.
The program
\program{experiments/interleaved_rank_distributions.py}{interleaved_rank_distributions.py}
extends orders one vertex at a time. Write the prefix rows of $N_\pi$ as
$(P\mid Q)$, with $P$ on the prefix columns and $Q$ on the unused vertices $R$.
Group prefixes by $R$, $\rank P$, and $W=\{xQ:xP=0\}$.
Appending $v\in R$ raises $\rank P$ by one exactly when some $w\in W$ has $w_v=1$.
On $R'=R\setminus\{v\}$ the new space is spanned by the restrictions of the
$w\in W$ with $w_v=0$ and the adjacency row of $v$ on $R'$.
The program carries each unused set to an orbit representative under graph
automorphisms, permutes its subspace by the same automorphism, and adds multiplicities.

\begin{table}[htb]
\centering\small
\caption{Boundary ranks over $\mathbb F_2$ after insertion of isolated vertices.
Each column counts all orders of the original vertices. Omitted ranks have frequency zero.}
\label{tab:interleaved-ranks}
\begin{tabular}{r|rrr}
Rank of $\partial A_\pi$ & Petersen & $4\times4$ rook & Shrikhande \\ \hline
10 & 491,760 & 0 & 0\\
12 & 1,411,440 & 337,305,600 & 56,217,600\\
14 & 1,542,480 & 6,880,444,416 & 940,732,416\\
16 & 182,880 & 97,908,092,928 & 10,947,044,544\\
18 & 240 & 686,780,916,480 & 99,538,907,904\\
20 & 0 & 2,664,866,534,400 & 707,161,639,680\\
22 & 0 & 5,946,216,414,336 & 3,206,911,138,752\\
24 & 0 & 7,325,845,130,880 & 7,994,269,948,608\\
26 & 0 & 4,058,989,516,800 & 8,666,847,059,520\\
28 & 0 & 134,943,137,280 & 236,094,968,448\\
30 & 0 & 22,394,880 & 22,230,528\\ \hline
Total & $10!$ & $16!$ & $16!$
\end{tabular}
\end{table}

The rook and Shrikhande graphs share the strongly regular parameters
$(16,6,2,2)$ and adjacency spectrum $6^1,2^6,(-2)^9$
\cite{Shrikhande1959}. The program \program{experiments/interleaved_graph_paths.py}{interleaved_graph_paths.py}
checks $C_\pi^2=4I+2J$ over $\mathbb Z$ for the original adjacency matrix of each graph. The different
frequencies at rank $30$ in Table~\ref{tab:interleaved-ranks} show that the adjacency
spectrum does not determine the distribution of $\rank(\partial A_\pi)$.
Counting Hamiltonian paths up to reversal, Wolfram's \texttt{GraphData}
\cite{WolframGraphData} gives $120$, $11{,}197{,}440$ and $11{,}115{,}264$
for the Petersen, $4\times4$ rook and Shrikhande graphs, respectively.
Doubling these counts gives the three full-rank frequencies in the table.

\section{Horns on rooted trees}\label{sec:tree-horns}
A morphism $f:S\to T$ of rooted trees in $\Omega$ is a morphism of their
free symmetric coloured operads \cite[\S3]{MoerdijkWeissDendroidal}.
Its colour map is $E(f):E(S)\to E(T)$ on edges. For finite sets $C_a$, put
$I(T)=\prod_{a=1}^k(E(T)\sqcup C_a)$ and $I(f)=\prod_a(E(f)\sqcup 1_{C_a})$.
Precomposition defines $D:\Omega^{\mathrm{op}}\to\mathbf{Mod}_A$, with $D(T)=A^{I(T)}$.
Number the edges of the linear tree $L_p$ from $0$ at the leaf to $p$ at the root.
Morphisms of linear trees then act by ordinal maps. Numbering each $C_a$ in a
fixed order from $p+1$ to $p+c_a$, where $c_a=|C_a|$, gives $D|_\Delta\cong X^{\mathbf c}(A)$.

Contracting an inner edge $e$ gives an elementary inner face $T/e\to T$.
Pruning a vertex incident to just one inner edge gives an elementary outer
face and retains that edge. The elementary faces of a one-vertex tree are its
edge inclusions. Regard faces as subobjects
of $\Omega[T]=\Omega(-,T)$. The horn omitting $F:S\to T$ is the union of the
other elementary faces $G:S_G\to T$. A map from the horn to the underlying
dendroidal set of $D$ is a family $z_G\in D(S_G)$ with $i^*z_G=j^*z_{G'}$
on each common subface $Gi=G'j$. A morphism $u$ in an intersection factors as
$m\sigma$, with $m$ a face up to isomorphism and $\sigma$ a composite of
degeneracies \cite[Lemma~3.1]{MoerdijkWeissDendroidal}. Degeneracies have
sections, so $m=us$ for a section $s$ of $\sigma$ and $m$ factors through both faces.
Write $M_{T,F}$ for the module of compatible families. Fillers are preimages
under $\rho_{T,F}:D(T)\to M_{T,F}$. Put $K_F(T)=\ker\rho_{T,F}$ and
$K(T)=\bigcap_G\ker G^*$ over all elementary faces. Identify $I(S_G)$
with its image in $I(T)$.

\begin{proposition}\label{prop:tree-lifts}
Let $F:S\to T$ be an elementary face and $U=\bigcup_{G\ne F} I(S_G)$.
The restriction $F^*:K_F(T)\to K(S)$ has kernel $K(T)$ and image the
tensors supported on $I(S)\setminus U$. Extension by zero to $I(T)$
splits the map onto its image.
\end{proposition}
\begin{proof}
Every elementary subface of $F$ factors through an included face
\cite[Lemma~2.1]{BasicNikolaus}, so $F^*K_F(T)\subseteq K(S)$.
Since $K_F(T)$ consists of tensors vanishing on $U$, restriction to $I(S)$
has the stated kernel, image and section.
\end{proof}

For the tree $a\to e$, $(e,b)\to r$, omit the contraction $F$ of $e$.
The included faces have colours $\{a,e\}$ and $\{e,b,r\}$. Their index
sets intersect in that of the single-edge face on $e$, so each compatible
family extends by zero off their union. With $k=2$ and $C_1=C_2=\varnothing$, $\rho_{T,F}$
has rank $2^2+3^2-1=12$ and kernel rank $4$. The six off-diagonal matrix units
span $K(S)$ on the contracted corolla $(a,b)\to r$. Of these six pairs,
only $(b,r)$ and $(r,b)$ occur in an included face. Thus $F^*:K_F(T)\to K(S)$
is injective with cokernel freely generated by the classes of $E_{(b,r)}$ and
$E_{(r,b)}$ (\program{experiments/two_vertex_tensor_horn.py}{two_vertex_tensor_horn.py}).

\begin{theorem}\label{thm:tree-cutoff}
Let $T$ have $n$ vertices and let $F$ be an elementary face.
For any finite sets $C_1,\ldots,C_k$, the horn restriction $\rho_{T,F}$ is
injective if $n\geq k+1$ and is an isomorphism if $n\geq k+2$.
For each $k\geq2$ there is a tree with $k+1$ vertices and an inner face $F$
such that $\operatorname{coker}\rho_{T,F}\cong A^{k!}$ for all $C_a$.
For $k=1$ every inner horn fills uniquely. For $k=1$ and $n=2$, $\rho_{T,F}$ is
an isomorphism unless $F$ prunes a nullary vertex and $C_1\ne\varnothing$,
in which case its cokernel is $A^{C_1}$.
\end{theorem}
\begin{proof}
For each $x\in I(T)$ form a graph on the included faces whose index sets
contain $x$. Join two if a common subface also contains $x$ in its index set.
A compatible family has one value at $x$ on each of the $c_x$ components.
Thus $\rho_{T,F}$ is the product of the diagonals $A\to A^{c_x}$.
Each $x$ with $c_x=0$ contributes $A$ to the kernel, and each $x$ with
$c_x>0$ contributes $A^{c_x-1}$ to the cokernel.

Suppose $n\geq3$. The $n-1$ inner edges join the
vertices in a tree with $o\geq2$ endpoints, one per outer face.
Each contraction deletes its inner edge, each pruning the outer edges at
its vertex. Thus the $n+o-2$ included faces delete pairwise disjoint
sets of colours. If $c_x=0$, the tuple $x$ meets all $n+o-2\geq n$ sets,
so $k\geq n$.

Two contractions, two prunings, or a pruning and a nonincident contraction
commute. Their common face retains every shared colour. Only a pruning
and its incident contraction can fail to be adjacent. These pairs are
disjoint for $n\geq3$, so a third included face containing $x$ would join both.
If $c_x\geq2$, the only included faces containing $x$ are the pruning of an endpoint
$v$ and the contraction of its incident edge $e$.

Choose a colour of $x$ deleted by each of the other $n+o-4$ included faces.
These colours are distinct, so $k\geq n-1$ if $o\geq3$. If $o=2$, the vertices form a path. Pruning
$v$ and then its neighbour gives the same face as contracting $e$ and pruning
the merged vertex. This common face retains the other inner edges, the
outer edges at the opposite endpoint and all $C_a$. It contains the $n-2$
colours just counted, so disconnection requires a further tree edge in $x$.
Again $k\geq n-1$.
Thus every $c_x=1$ when $n\geq k+2$.

For $k\geq2$, take a chain of $k+1$ unary vertices and give the second vertex
from the top an extra leaf $b$. Omit the contraction of the inner edge $h$
at the root vertex. This removes the contraction paired with the root pruning.
Hence if $c_x\geq2$, the two included faces containing $x$ are the top pruning
and the contraction of the inner edge $e$ at the top vertex.
Every other included face must delete an entry of $x$, so $x$ contains the
root edge $r$ and the $k-2$ inner edges other than $e,h$. Pruning the top
two vertices retains every shared colour except $b$, so $x$ also contains $b$.

Each permutation of these $k$ edges occurs in exactly those two included faces. A common
subface containing $b$ and $r$ would have leaf $b$ and root $r$. Composing
its vertices would give an operation with output $r$ and $b$ among its inputs.
In $T/e$ the other input of the vertex at $b$ is the top leaf, which the
pruning removes. Thus each permutation contributes $A^2/\mathrm{diag}(A)\cong A$
to the cokernel, with coordinate the difference of its two face values.

For $k=1$ and two vertices, any two included faces cover the edges.
Restriction to a shared edge identifies its two labels and the two functions
on $C_1$. The prunings share the inner edge. The contraction shares
the root edge with the pruning of the upper vertex, and the inputs of
the upper vertex with the pruning of the root vertex. Only when $F$ prunes
a nullary upper vertex is there no common edge and hence no common subface.
Restriction on $C_1$ is then the diagonal
$A^{C_1}\to A^{C_1}\times A^{C_1}$, with cokernel $A^{C_1}$.
\end{proof}

Every horn on $L_{k+1}$ fills uniquely by Theorem~\ref{thm:hornbasis}.
For $k=2$ the branching example is $a\to e$, $(e,b)\to h$, $h\to r$, with
the contraction of $h$ omitted. For empty $C_a$ the horn restriction is
$A^{25}\hookrightarrow A^{27}$. Its two cokernel coordinates are the differences
between the values on the top pruning and on the contraction of $e$ at $(b,r)$ and $(r,b)$
(\href{https://github.com/flengyel/SimplicialTensors/blob/30478d42b758cd0a559c72ac6535725055e2e59e/experiments/three_vertex_tensor_horn.py}{\nolinkurl{three_vertex_tensor_horn.py}}).

For $k=1$ the sharp bound for unique filling of all horns is two vertices if
$C_1=\varnothing$ and three otherwise, since a horn on $L_1$ leaves one edge label free.
Take the operad with colours $(a,\lambda)\in A\times A^{C_1}$ and one operation
for each list of inputs and output sharing $\lambda$, and none otherwise.
A labelling of $T$ amounts to a function $E(T)\to A$ and a single $\lambda$,
since $T$ is connected. Under precomposition these bijections identify the
nerve of the operad with the underlying dendroidal set of $D$.
Moerdijk and Weiss characterize dendroidal nerves of operads by unique filling
of every inner horn \cite[Proposition~5.3 and Theorem~6.1]{MoerdijkWeissInnerKan}.
For $k\geq2$ the unfillable inner horn in Theorem~\ref{thm:tree-cutoff} therefore
excludes an operadic nerve.

\appendix
\section{Kernel families in a product of Grassmannians}
\label{app:kernel-families}
Fix a field $K$, a degree $n$ and excesses $\mathbf c$ as in
Section~\ref{sec:generated}. Let $W$ be the affine $K$-space of degree-$n$
tensors, so that $W(F)=X^{\mathbf c}_n(F)$ for every field extension $F/K$.
Write $C_F$ and $\mathcal D_F$ for the scalar extensions of $C=N(K[\Delta^n])$
and $\mathcal D=N(X^{\mathbf c}(K))$. By representability and Dold--Kan,
$T\mapsto g_T$ identifies $W(F)$ with the space of chain maps $C_F\to\mathcal D_F$.

Let $r_p$ be the rank of $g_{T,p}$ at the generic point of $W$,
and put $\ell_p=\dim C_p-r_p$. The locus $U$ where these ranks are attained
simultaneously is dense and open. On $U$ the kernels define a morphism to
the product of Grassmannians of $\ell_p$-planes in $C_p$,
\begin{equation}\label{eq:kernel-map}
 \Psi:U\longrightarrow G=\prod_{p=0}^n\operatorname{Gr}(\ell_p,C_p),
 \qquad T\longmapsto (\ker g_{T,p})_p.
\end{equation}
The image of $\Psi$ in the underlying space of $G$ is constructible by
Chevalley's theorem \cite[Tag 054K]{Stacks}. Its reduced closure $Y$ is integral since $U$ is irreducible.

\begin{proposition}\label{prop:kernel-rationality}
Over any field $K$, the reduced closure of the kernel family $\Psi(U)$ is
birational to an affine space over $K$.
\end{proposition}
\begin{proof}
Let $\eta$ be the generic point of $W$ and $F=K(W)$. By
\eqref{eq:finite-chain-splitting}, $\im g_\eta$ is isomorphic to $Q_F$ for a
complex $Q$ over $K$ with the same disk and sphere multiplicities.
By the proof of Proposition~\ref{prop:kernel-realizability}, cycle and boundary
dimensions determine whether one complex embeds in another. Scalar extension
preserves these dimensions, so $Q_F\hookrightarrow\mathcal D_F$ gives an embedding
$Q\hookrightarrow\mathcal D$ over $K$. Applied to the duals with degrees reversed,
the same criterion gives $C\twoheadrightarrow Q$ from $C_F\twoheadrightarrow Q_F$.
The composite $C\to Q\to\mathcal D$ is $g_{T_0}$ for some $T_0\in U(K)$.

In $W\times Y$ let $\mathcal E$ be defined by the incidence equations
$g_{T,p}(L_p)=0$ on the tautological subbundles $L_p$.
Over $U$ the inclusion $L_p\subseteq\ker g_{T,p}$ is an equality of subbundles
of the same rank, so $\mathcal E\cap(U\times Y)$ is the graph of $\Psi$.
At $y_0=\Psi(T_0)$ the solutions are the chain maps $Q\to\mathcal D$,
since $C/\ker g_{T_0}\cong Q$. At $\Psi(\eta)$ the solutions over $F$ are the
chain maps $Q_F\to\mathcal D_F$. The solution spaces have the same dimension, so
the incidence equations have their generic rank $m$ at $y_0$.

Use the tensor coordinates on $W$ and choose local bases of the tautological
subbundles near $y_0$. Some minor of order $m$ is nonzero at $y_0$.
Let $Y^\circ$ be the open set where it is invertible. Since $Y$ is integral, all larger
minors vanish there. Projection $\pi$ onto the $d=\dim W-m$ coordinates
outside the minor's columns identifies $\mathcal E|_{Y^\circ}$ with
$Y^\circ\times\mathbb A_K^d$ by solving the linear equations in $T$.
Intersect the open graph with $\pi(T)=\pi(T_0)$. This identifies an open
neighborhood of $T_0$ in the affine $m$-plane $\pi^{-1}(\pi(T_0))$
with an open neighborhood of $y_0$ in $Y$. In particular,
$\dim Y=m=\dim W-\dim_K\operatorname{Hom}_{\mathrm{Ch}}(Q,\mathcal D)$.
\end{proof}

Over $\mathbb C$, $Y$ is isomorphic to a Schubert variety.
The subcomplexes $L\subseteq C$ with $\dim L_p=\ell_p$ form a quiver Grassmannian.
By \eqref{eq:simplex-projective}, every indecomposable summand of $C$ has
dimension one or two. Cerulli Irelli, Feigin and Reineke embed the quiver
Grassmannian in a partial flag variety. For such a representation their
construction gives $C_p=V_p/S_p$ with both flags $V_\bullet$ and
$S_\bullet$ preserved by a common Borel subgroup $B$
\cite[Example~1.3, Proposition~2.1 and Theorem~2.2]{SchubertQuiver}.
The inclusions $V_p\subseteq V_{p-1}$
induce the differentials, so each $b\in B$ induces a chain automorphism
$u_b$ of $C$. Precomposition by $u_b^{-1}$ preserves the ranks of
$g_T$ and sends its kernel to $u_b(\ker g_T)$. Hence $Y$ is $B$-stable.
The $B$-orbits of the flag variety are its finitely many Schubert cells.
Since $Y$ is closed and irreducible, it is the closure of one such cell.

\begin{proposition}\label{prop:three-by-three}
Let $n=2$ and $\mathbf c=(0,0)$, so $T$ is a $3\times3$ matrix. On $U$
the kernel dimensions in degrees $0,1,2$ are $2,0,0$. Identify a plane
$\ker(a,b,c)$ with $[a:b:c]$ in the dual projective plane $G$. For every
field extension $F/K$, the image of $U(F)$ is $G(F)\setminus\{[1:1:1]\}$.
The chain automorphisms of $C_F$ act transitively on this image.
For each $T\in U(F)$, $H_1(\langle T\rangle)\cong F$ and
$H_j(\langle T\rangle)=0$ for $j\ne1$.
\end{proposition}
\begin{proof}
Over $F$, write $T=(t_{ij})$ and $(a,b,c)=(t_{00},t_{11},t_{22})$.
In degree zero $g_{T,0}$ is the row $(a,b,c):F^3\to F$.
In degree one, normalization subtracts $t_{ii}$ from every entry of
the restriction to $i,j$, for $i<j$. In the basis $(E_{01},E_{10},E_{11})$
of $N_1(X)$ the normalized restrictions to $01,02,12$ are the columns of
\[
 B_T=\begin{pmatrix}
 t_{01}-a&t_{02}-a&t_{12}-b\\
 t_{10}-a&t_{20}-a&t_{21}-b\\
 b-a&c-a&c-b
 \end{pmatrix}.
\]
For fixed $a,b,c$ the first two rows are arbitrary. They can complete the
last row to a basis exactly when $a,b,c$ are not all equal. If $\det B_T\ne0$,
then $g_{T,0}$ has rank one and $g_{T,1}$ is injective.
By \eqref{eq:simplex-projective} the differential of $C_F$ is injective
in degree two, so the chain map identity gives injectivity of $g_{T,2}$.
Thus the maximal ranks are $(1,3,1)$ and $U=\{T:\det B_T\ne0\}$.
The kernel tuple is $(\ker(a,b,c),0,0)$. The row construction works over
every $F$, so it gives the asserted image and $Y=G$. The fiber over each such
plane is the open set $\det B_T\ne0$ in the seven-dimensional space of matrices
whose diagonals lie on the line spanned by $(a,b,c)$.

The augmentation $(1,1,1)$ restricts to a nonzero linear form on
$\ker(a,b,c)$, with one-dimensional kernel. The homology follows from
\eqref{eq:generated-homology}. Each quotient has chain dimensions $(1,3,1)$
and differential ranks $(0,1,1)$, so they are all isomorphic by \eqref{eq:finite-chain-splitting}.
Proposition~\ref{prop:quotient-lift} gives the orbit assertion.
\end{proof}

Over $\mathbb F_q$ there are $q^2+q$ such planes. For each one, $q-1$
diagonals represent its normal. Given the nonzero last row of $B_T$, there
are $q^3-q$ choices for the first row outside its span, then $q^3-q^2$
choices for the second outside the span of the other two. Hence each plane
occurs for $(q-1)(q^3-q)(q^3-q^2)$ matrices.
For $q=2$, \program{experiments/three_by_three_kernels.py}{three_by_three_kernels.py}
computes the kernel of each of the $512$ matrices and finds $24$ above each
of six planes, $144$ in all. Over any $K$, the all-ones matrix lies outside $U$. Its kernel is
$(\ker(1,1,1),C_1,C_2)$, and it generates the constant simplicial module with homology
$K$ in degree zero alone.

\medskip
\noindent\textbf{Acknowledgment.}
The author used OpenAI's ChatGPT and Codex to develop and check
mathematical arguments, write and test code for the computations,
and revise the manuscript.

\small
\AddToHookNext{cmd/bibitem/before}{\setlength{\itemsep}{1pt plus 0.5pt minus 0.5pt}}
\bibliographystyle{plain}
\bibliography{dstm_short}
\end{document}